\documentclass[12pt]{article}
\usepackage{amsthm}
\usepackage{amsmath}
\usepackage{amsfonts}
\usepackage{graphicx}
\usepackage{latexsym}
\usepackage{amssymb}
\usepackage{lmodern}
\usepackage{color}

\newcommand{\abs}[1]{\left|#1\right|}

\newcommand{\field}[1]{\mathbb{#1}}
\newcommand{\A}{\field{A}}
\newcommand{\B}{\field{B}}
\newcommand{\C}{\field{C}}
\newcommand{\Z}{\field{Z}}

\newcommand{\cA}{{\cal A}}
\newcommand{\cB}{{\cal B}}
\newcommand{\cC}{{\cal C}}

\newcommand{\cF}{{\cal F}}

\newcommand{\cH}{{\cal H}}
\newcommand{\cG}{{\cal G}}

\newcommand{\blda}{{\mathbf{a}}}
\newcommand{\bldb}{{\mathbf{b}}}
\newcommand{\bldc}{{\mathbf{c}}}
\newcommand{\blde}{{\mathbf{e}}}
\newcommand{\bldu}{{\mathbf{u}}}
\newcommand{\bldv}{{\mathbf{v}}}
\newcommand{\bldx}{{\mathbf{x}}}
\newcommand{\bldy}{{\mathbf{y}}}
\newcommand{\bldz}{{\mathbf{z}}}

\newcommand{\support}{{\mathsf{supp}}}

\newcommand{\distance}{{\mathsf{d}}}
\newcommand{\weight}{{\mathsf{wt}}}
\newcommand{\bfzero}{{\bf 0}}

\newcommand{\wt}[1]{ \text{wt} ({#1}) }

\newtheorem{defn}{Definition}
\newtheorem{theorem}{Theorem}
\newtheorem{lemma}{Lemma}

\newtheorem{corollary}{Corollary}

\begin{document}

\bibliographystyle{plain}

\title{
\begin{center}
Maximum Size $t$-Intersecting Families and Anticodes
\end{center}
}
\author{
{\sc Xuan Wang}\thanks{Department School of Mathematical Sciences, Anhui University, Hefei 230601, China,
e-mail: {\tt wang$\_$xuan$\_$ah@163.com}.} \and
{\sc Tuvi Etzion}\thanks{Department of Computer Science, Technion, Israel Institute of Technology, Haifa 3200003, Israel,
e-mail: {\tt etzion@cs.technion.ac.il}.} \and
{\sc Denis S. Krotov}\thanks{Sobolev Institute of Mathematics, Novosibirsk 630090, Russia,
e-mail: {\tt krotov@math.nsc.ru}.} \and
{\sc Minjia Shi}\thanks{Department School of Mathematical Sciences, Anhui University, Hefei 230601, China,
e-mail: {\tt smjwcl.good@163.com}.}}

\maketitle

\begin{abstract}
The maximum size of $t$-intersecting families is one of the most celebrated topics in combinatorics,
and its size is known as the Erd\H{o}s-Ko-Rado theorem. Such intersecting families, also
known as constant-weight anticodes in coding theory, were considered in a generalization of the well-known sphere-packing bound.
In this work we consider the maximum size of $t$-intersecting families and their associated
maximum size constant-weight anticodes over alphabet of size $q >2$. It is proved that the structure of the maximum size
constant-weight anticodes with the same length, weight, and diameter, depends
on the alphabet size. This structure implies some hierarchy of constant-weight anticodes.
\end{abstract}

\vspace{0.5cm}



\newpage
\section{Introduction}
\label{sec:introduction}

A system $\cA$ of $k$-subsets of an $n$-set is {\bf \emph{$t$-intersecting}} if
$$
\abs{A_1 \cap A_2} \geq t ~~ \text{for~all} ~~ A_1,A_2 \in \cA
$$
Such a system is called a {\bf \emph{$t$-intersecting family}}.

Finding the size of the largest system among all the $\binom{n}{k}$ $k$-subsets is one of the
most intriguing combinatorial problems, initiated by Erd\"{o}s-Ko-Rado~\cite{EKR61}.
During the years, the problem for $t \geq  1$ was considered and many
interesting results were found, e.g.~\cite{FrFu80,FrTo99}. The intersecting problem is also considered for an $n$-set,
but with unrestricted subsets instead of $k$-subsets.
Wilson~\cite{Wil84} gave an exact bound for the Erd\"{o}s-Ko-Rado theorem.
\begin{theorem}
\label{thm:Wilson}
If $n \geq (t+1)(k-t+1)$ then any $t$-intersecting family of $k$-subsets from an $n$-set contains at most $\binom{n-t}{k-t}$ subsets.
The bound is attained by all $k$-subsets (of an $n$-set) that contain a fixed $t$-subset.
If $n > (t+1)(k-t+1)$ this family is unique and if $n = (t+1)(k-t+1)$ there is another family with the same parameters and
the same size.
\end{theorem}

The bound of Theorem~\ref{thm:Wilson} was already proved in~\cite{EKR61} for $t=1$. A survey for the large amount of research associated
with this problem for the first 22 years was given by Deza and Frankl~\cite{DeFr83}.

The intersection problem was completely solved by Ahlswede and Khachatrian~\cite{AhKh97}, when they considered $n < (t+1)(k-t+1)$.
Later Ahlswede and Khachatrian~\cite{AhKh98} observed that the intersection problem is strongly connected to
the diametric problem in the Hamming spaces. The diametric problem is an important problem in coding theory and it was considered
in various metric spaces in many papers~\cite{BuEt15,Etz11,Etz22,EGRW16,SXK23,TaSc10}.
The problem has found an application also in coding theory as was first observed by Delsarte~\cite{Del73}.
A {\bf \emph{code}} of length $n$ is a set of words over some alphabet whose length is $n$.
A $t$-intersecting family is called an anticode (a type of code) and it was used to improve the well-known sphere-packing bound with
a bound called the code-anticode bound.
This bound was used later by Roos~\cite{Roo83} to eliminate the possible existence of perfect codes in the Johnson scheme
for various parameters. Ahlswede, Aydinian, and Khachatrian~\cite{AAK01} defined the concept
of diameter perfect code based on this bound. They have discussed diameter perfect codes in the Hamming and the Johnson scheme, and mentioned
also the Grassmann scheme.

\begin{defn}
For two words $\bldx=(\bldx_1,\bldx_2,\ldots,\bldx_n)$ and $\bldy=(\bldy_1,\bldy_2,\ldots,\bldy_n)$, over an alphabet $\Sigma_q$ with
$q \geq 2$ letters, the Hamming distance $\distance (\bldx,\bldy)$ is the number of coordinates in which $\bldx$ and $\bldy$ differ, i.e.,
$$
\distance (\bldx,\bldy) \triangleq \abs{\{ i ~:~ \bldx_i \neq \bldy_i, ~ 1 \leq i \leq n \}}.
$$

W.l.o.g. (without loss of generality) we assume that our alphabet $\Sigma_q$ is $\Z_q$ which contains the integers
in the set $\{ i ~:~ 0 \leq i \leq q-1\}$.

The {\bf \emph{weight}}, $\weight(\bldx)$, of a word $\bldx=(\bldx_1,\bldx_2,\ldots,\bldx_n)$, over an alphabet $\Z_q$
with $q \geq 2$ letters, is the number of nonzero entries in $\bldx$, i.e.,
$$
\weight(\bldx) \triangleq \abs{ \{i ~:~ \bldx_i \neq 0, ~ 1 \leq i \leq n \} }.
$$
In other words, the weight of $\bldx$ is its distance from the allzero word, $\bfzero$, i.e., $\weight(\bldx)=\distance(\bldx,\bfzero)$.
\end{defn}

Let $\cF$ be a set of $k$-subsets of an $n$-set. When the $k$-subsets of $\cF$ are represented by words of length $n$,
the outcome is a binary code $\cC$ whose codewords have constant-weight $k$.

The minimum distance of the code $\cC$ is the least distance between any two distinct codewords of $\cC$.
The constant-weight code $\cC$ is referred to as an $(n,d,k)_q$ code, where $n$ is the length of the codewords, $k$ is their weight,
and $d$ is the minimum distance of $\cC$. If $\cC$ is a code associated with a $t$-intersecting family, then
we are interested in the maximum distance $D$ between any codewords in $\cC$.
The maximum distance of the code is called the {\bf \emph{diameter}} of $\cC$ and such a set $\cC$ is called an {\bf \emph{anticode}}.
A constant-weight anticode with diameter $D$, codewords of length $n$ and weight $k$ will be referred to as an $(n,D,k)_q$ anticode.
The code-anticode bound was given first by Delsarte~\cite{Del73} and was further discussed
in~\cite{AAK01} for schemes based on distance-regular graphs. In this paper we are interested in constant-weight codes
over $\Z_q$, where $q>2$. In this case
the metric is not based on a distance-regular graph. The proof for this version
is presented in~\cite{Etz22,Etz22b}.
\begin{theorem}
\label{thm:code-anticode}
If $\cC$ is an $(n,d,k)_q$ code and $\cA$ is an $(n,d-1,k)_q$ anticode, then
\begin{equation}
\label{eq:code_anti_bound}
\abs{\cC} \cdot \abs{\cA} \leq \binom{n}{k}(q-1)^k .
\end{equation}
\end{theorem}

\begin{defn}
The bound in Equation~\textup{(\ref{eq:code_anti_bound})} is the {\bf \emph{code-anticode bound}}.
A code $\cC$ which satisfies Equation~\textup{(\ref{eq:code_anti_bound})} with equality is called a
{\bf \emph{diameter perfect constant-weight code}}.
The anticode~$\cA$ is called a {\bf \emph{maximum size anticode}}.
\end{defn}

Equation~\textup{(\ref{eq:code_anti_bound})} is the motivation to consider constant-weight anticodes (as well as their combinatorial interest)
over $\Z_q$, $q>2$.
The goal of the current work is to consider $t$-intersecting families and anticodes over a non-binary alphabet, where words have length $n$ and
weight $k$.

\begin{defn}
\label{dfn:ktq_intersect}
A {\bf \emph{ $(t,k)_q$-intersecting family}} $\cF$ is a set of words of length $n$ over $\Z_q$ such that for each
two words $\bldx=(\bldx_1,\bldx_2,\ldots,\bldx_n)$ and $\bldy=(\bldy_1,\bldy_2,\ldots,\bldy_n)$ in $\cF$ the number of equal nonzero
entries is at least $t$, i.e.,
$$
\abs{ \{i ~:~ \bldx_i = \bldy_i ~~ \textup{and} ~~ \bldx_i \neq 0, ~ 1 \leq i \leq n \} } \geq t .
$$
\end{defn}

There is another definition for $(t,k)_q$-intersecting family, different from
Definition~\ref{dfn:ktq_intersect}.  In this definition, e.g.~\cite{AAK01,AhKh98}, the
requirement $\bldx_i \neq 0$ is dropped, i.e., the intersection includes coordinates where the two words have \emph{zeros}.
We use Definition~\ref{dfn:ktq_intersect} since it is also a generalization of $k$-intersecting family in the binary case
and moreover optimal $(t,k)_q$-intersecting families defines in Section~\ref{sec:intersect} are optimal constant-weight anticodes with small diameter.
Definition~\ref{dfn:ktq_intersect} yields small diameter when $t$ and $k$ are fixed, while $n$ is large as it is required here.
In addition, it is interesting to note that by using Definition~\ref{dfn:ktq_intersect} a $(t,k)_q$-intersecting family
is equivalent to a $t$-intersecting family in the Johnson scheme
$J(n(q-1),k)$ where we can choose only the transversal $k$-subsets,
with respect to the fixed partition of the base $n(q-1)$-set into $n$ subsets of size $q-1$.
These $t$-intersecting families in $J(n(q-1),k)$ are anticodes with diameter $k-t$ with the Johnson distance (which is half of the Hamming distance).
However, when $k=n$, i.e., there are no \emph{zeros} in the words, the two definitions coincide and the following results
were obtained for this case.

The first important result for $k=n$ was proved by Frankl and F\"{u}redi~\cite{FrFu80}.
\begin{theorem}
\label{thm:FFintersect}
If $\cF$ be a maximum size $(t,k)_q$-intersecting family with words of length $n$, where $k=n$ and $t \geq 15$, then
$$
\abs{\cF} = (q-1)^{n-t} ~~ \text{if~and~only~if} ~~ q \geq t+1 .
$$
\end{theorem}
The result of Theorem~\ref{thm:FFintersect} was improved later by Frankl and Tokushige~\cite{FrTo99}.
\begin{theorem}
\label{thm:bestFrankl}
If $\cF$ be a maximum size $(t,k)_q$-intersecting family with words of length $n$, where $k=n$ and $q \geq t+1$, then
$$
\abs{\cF} = (q-1)^{n-t} .
$$
\end{theorem}

The results of Theorem~\ref{thm:bestFrankl} was proved in parallel by Ahlswede and Khachatrian~\cite{AhKh98}, who also solved the
remaining cases when $q < t+1$.
\begin{corollary}
\label{cor:max_ant_nr}
If $q \geq t+1$, then the maximum size anticode of length $n$, and diameter $n-t$ over~$\Z_q$ is $q^{n-t}$.
\end{corollary}
Similarly to Corollary~\ref{cor:max_ant_nr}, the diameter of an anticode by the definition of~\cite{AAK01,AhKh98}, i.e.,
when the \emph{zeros} are considered in the intersection of size $t$ is $n-t$. This is not the case if we consider
Definition~\ref{dfn:ktq_intersect} which is the definition used in this our work.

The rest of the paper is organized as follows.
In Section~\ref{sec:intersect} we define one intersecting family and two anticodes,
prove asymptotic optimality of these sets and the uniqueness of their structure.
The defined sets of anticodes are generalized in Section~\ref{sec:hierarchy},
where two sequences of such constant-weight anticodes are defined.
This generalization induces a hierarchy between the anticodes in each sequence.
The hierarchy for anticodes is based on their size, where each one is
larger for different range of alphabet size and length. Two anticodes in this hierarchy are compared only when
they have the same alphabet size, length, weight, and diameter. This hierarchy is analyzed in Section~\ref{sec:hierarchy}.
Finally, in Section~\ref{sec:conclusion} conclusion, comparison with other maximum size anticodes,
and directions for future research are presented and discussed.

\section{Maximum Size $t$-Intersecting Families and Anticodes}
\label{sec:intersect}

In this section we show asymptotic maximality of one $(t,k)_q$-intersecting family for any admissible triple $(t,k,q)$ and two
families of anticodes for any admissible triple $(D,k,q)$, where $D$ is the diameter of the anticode, $q$ is the alphabet size, and $k$ is
the constant weight of the anticode. We also prove the uniqueness of the intersecting families and the anticodes.
The two families of the anticodes differ in the parity of the diameter, even or odd.

We start with a $(t,k)_q$-intersecting family defined by
$$
\cF_q (t,k,n) \triangleq \{ ( \overbrace{1\cdots \cdots 1}^{t ~ \text{times}}   \bldb_1 \cdots \bldb_{n-t}) ~:~
\bldb_i \in \Z_q , ~ \weight{(\bldb_1 \cdots \bldb_{n-t})}=k-t \}
$$

\begin{lemma}
\label{lem:propF}
If $n \geq k >t \geq 0$ then $\cF_q (t,k,n)$ is a $(t,k)_q$-intersecting family of length $n$
and $\binom{n-t}{k-t} (q-1)^{k-t}$ words.
\end{lemma}
\begin{proof}
It is readily verified that $\cF_q (t,k,n)$ is a $(t,k)_q$-intersecting family of length $n$.
The size of $\cF_q (t,k,n)$ follows immediately from choosing the $k-t$ nonzero coordinates in the last $n-t$ coordinates and each one
of these coordinates can be assigned with $q-1$ possible values.
\end{proof}

The {\bf \emph{support}}, $\support (\bldx)$, of a word $\bldx =(\bldx_1,\bldx_2,\ldots,\bldx_n)$ is the set of coordinates with values
different from zero, i.e.,
$$
\support (x) \triangleq \{ i ~:~ \bldx_i \neq 0, ~ 1 \leq i \leq n \}
$$

\begin{theorem}
\label{thm:asymF}
If $n \geq k >t \geq 0$, $n \geq (t+1)(k-t+1)$, and $q \geq t+1$, then the $(t,k)_q$-intersecting family $\cF_q (t,k,n)$ is
a maximum size $(t,k)_q$-intersecting family for a given $q > 2$.
\end{theorem}
\begin{proof}
Let $\cG$ be a $(t,k)_q$-intersecting family with words of length $n$ and let $\cH$ be
the set of supports of the words in $\cG$, i.e.,
$$
\cH \triangleq \{ \support (\bldx) ~:~ \bldx \in \cG \}~.
$$
The set $\cH$ is a $t$-intersecting family since each two words of $\cG$ are from a $(t,k)_q$-intersecting family, i.e.,
they have the same nonzero values in at least $t$ coordinates.
Hence, by Theorem~\ref{thm:Wilson} we have $\abs{\cH} \leq \binom{n-t}{k-t}$.
Each $k$-subset $\bldc \in \cH$ is a support of words in $\cG$ which form a $(t,k)_q$-intersecting family in which $k=n$.
Therefore, by Theorem~\ref{thm:bestFrankl} the number of codewords whose support is $\bldc$ is at most $(q-1)^{k-t}$.
Hence, the number of
codewords in $\cG$ is at most $\binom{n-t}{k-t} (q-1)^{k-t}$, which is the size of $\cF_q(t,k,n)$ by Lemma~\ref{lem:propF}.
%
\end{proof}

Two codes $\cC_1$ and $\cC_2$ of length $n$ with constant-weight $k$ over $\Z_q$ are said to be {\bf \emph{equivalent}} if
$\cC_1$ can be obtained from $\cC_2$ by permuting rows, columns, and symbols. We say that maximum size intersecting family or
an anticode, with fixed parameters, over $\Z_q$ is {\bf \emph{unique}} if any two intersecting families or anticode, respectively, of maximum size with
the fixed parameters are equivalent.

\begin{theorem}
\label{thm:uniqueInter}
If $n \geq k >t > 1$, $n > (t+1)(k-t+1)$, and $q \geq t+1$ ($n >3k-2$ if $t=1$), then $\cF_q (t,k,n)$
is a unique maximum size $(t,k)_q$-intersecting family.
\end{theorem}
\begin{proof}
By the proof of Theorem~\ref{thm:asymF}, the $(t,k)_q$-intersecting family $\cG$ has maximum size $\binom{n-t}{k-t} (q-1)^{k-t}$
if the $t$-intersecting family $\cH$ has maximum size $\binom{n-t}{k-t}$. By Theorem~\ref{thm:Wilson} we have that $\cH$ is equivalent
to a code in which the codewords are all the words of weight $k$ that have \emph{ones} in the first $t$ coordinates.
Two codewords of $\cG$ whose supports do not intersect on the last $n-t$ coordinates must have the same values in the first $t$ positions.

Assume there exist two codewords $\bldu, \bldv \in \cG$ that have different values in the first $t$ coordinates and hence
their supports have some intersection
in the other $n-t$ coordinates. There exists a codeword $\bldz \in \cG$ whose support does not intersect
the support of $\bldu$ in the last $n-t$ coordinates and does not intersect the support
of $\bldv$ in the last $n-t$ coordinates. This codeword $\bldz$ must have the same values as $\bldu$ in the first $t$ coordinates
and the same values as $\bldv$ in the first $t$ coordinates, a contradiction since $\bldu$ and $\bldv$ have different values in the
first $t$ coordinates. Therefore, all the codewords of $\cG$ have the same values on the first $t$ coordinates, w.l.o.g. \emph{ones}.

Thus, the unique maximum size $(t,k)_q$-intersecting family whose size is $\binom{n-t}{k-t} (q-1)^{k-t}$,
contains all the $\binom{n-t}{k-t} (q-1)^{k-t}$ words of weight $k$ over $\Z_q$ with \emph{ones}
in the first $t$ coordinates.
\end{proof}

The requirement that $n > (t+1)(k-t+1)$ in Theorem~\ref{thm:uniqueInter} is due to the fact that for $n = (t+1)(k-t+1)$
there exist two types of optimal $t$-intersecting families (see Theorem~\ref{thm:Wilson}). The requirement $n >3k-2$ for $t=1$
is due to the fact that to have the three codewords $\bldu$, $\bldv$, and $\bldz$, it is required that $n-t \geq 3(k-t) -1$
(implied by the intersection of their supports).

When $n \geq 2k-t$, the family $\cF_q (t,k,n)$ can be viewed also as an anticode of length $n$ constant weight $k$, and diameter $D=2(k-t)$
(since two codewords of maximum distance have distinct nonzero entries in the last $n-t$ coordinates, where
the weight of each codeword is $k-t$). A constant-weight anticode is defined with the its diameter $D$ instead of $t$, and with this definition we have
that $t= \frac{2k-D}{2}$. It follows that the anticode will be the $t$-intersecting family $\cF_q (t,k,n)$. Therefore, we continue
with the parameter $t$ for the anticodes.
Surprisingly, the same proof as in Theorem~\ref{thm:asymF} cannot be used to prove that this anticode is of maximum size for large enough $n$.
The reason is that an $(n,D,k)_q$ anticode is not necessarily a $(t,k)_q$ intersecting family. However, we will prove
that this intersecting family is also of maximum size as an anticode. Two families of $(n,D,k)_q$ anticodes will be defined now,
one for even diameter $D$ and a second for odd diameter~$D$.



We start with anticodes having odd diameter $D$. We will have for these anticodes that  ${D=2(k-t)-1}$ which is an odd integer $D$, where
$t= \frac{2k-1-D}{2}$, and $k >t \geq 0$. Let

$$
\cA_q (t,k,n) \triangleq \{ ( \overbrace{1\cdots \cdots 1}^{t ~ \text{times}} \blda  \bldb_1 \cdots \bldb_{n-t-1}) ~:~
\blda \in \Z_q \setminus \{0\}, ~  \bldb_i \in \Z_q , ~ \weight{(\bldb_1 \cdots \bldb_{n-t-1})}=k-t-1 \}
$$

\begin{lemma}
\label{lem:propA}
If $n \geq 2k-t-1$ then $\cA_q (t,k,n)$ is an $(n,2(k-t)-1,k)_q$ anticode
with $\binom{n-t-1}{k-t-1} (q-1)^{k-t}$ codewords.
\end{lemma}
\begin{proof}
Since $n \geq 2k-t-1$ and the weight in the last $n-t-1$ coordinates is at most $k-t-1$, it follows that there exist
two codewords in $\cA_q (t,k,n)$ whose nonzero entries in the last $n-t-1$ coordinates are in disjoint coordinates.
If two such codewords have different values in $\blda$, then their distance is $2(k-t)-1$ which is the diameter of the anticode.
The size of $\cA_q (t,k,n)$ follows immediately by choosing the $k-t-1$ nonzero coordinates, each one and also $\blda$ can
be assigned with $q-1$ possible values.
\end{proof}

If $n < 2k-t-1$ then $\cA_q (t,k,n)$ is still an anticode of size $\binom{n-t-1}{k-t-1} (q-1)^{k-t}$, but its diameter is smaller than $2(k-t)-1$.

\begin{theorem}
\label{thm:asymA}
If $n$ is large enough, then $\cA_q (t,k,n)$ is a maximum size anticode for given $q > 2$, $k >t \geq 0$, and
diameter $2(k-t)-1$.
\end{theorem}
\begin{proof}
Let $\cB$ be an $(n,2(k-t)-1,k)_q$ anticode and let $\cC$ be the set of supports of the codewords in $\cB$, i.e.,
$$
\cC \triangleq \{ \support (\bldx) ~:~ \bldx \in \cB \}~.
$$
The code $\cC$ forms a $(t+1)$-intersecting family since any two words whose supports intersect in less than $t+1$ coordinates
have distance at least $2(k-t)$ and cannot belong to $\cB$.

We partition $\cC$ into two subsets
\begin{align*}
\cC_1 & = \{ \bldv \in \cC ~:~ \abs{\bldv \cap \bldu} = t+1 ~~ \text{for~some} ~~ \bldu \in \cC \}, \\
\cC_2 & = \{ \bldv \in \cC ~:~ \abs{\bldv \cap \bldu} \geq t+2 ~~ \text{for~all} ~~ \bldu \in \cC \}.
\end{align*}

We distinguish now between two cases depending whether $\cC_2$ is an empty subset or not.

\noindent
{\bf Case 1:} If $\cC_2 = \varnothing$ then $\cC=\cC_1$.
We further claim that for each support $\bldu \in \cC$ there are at most $(q-1)^{k-t}$ codewords in $\cB$ whose support is~$\bldu$.
By the definition of $\cC_1$ there exists a codeword $\bldv \in \cC$ such that $\abs{\bldu \cap \bldv}=t+1$.
The number of codewords of $\cB$ with the support $\bldu$, where the values of all these codewords in $\bldu \cap \bldv$
are equal is at most $(q-1)^{k-t-1}$,
which is the number of distinct assignments to the other $k-t-1$ coordinates. Let $\bldx , \bldy \in \cB$ be two codewords
such that $\bldu = \support (\bldx)$ and $\bldv = \support (\bldy)$. Since the diameter of $\cB$ is $2(k-t)-1$, and
$\abs{\bldu \cap \bldv} = t+1$, it follows that $\bldx$ and $\bldy$ have different values in at most one coordinate
of $\bldu \cap \bldv$.
If two codewords of $\cB$ whose support is $\bldu$ differ in more than two coordinates of $\bldu \cap \bldv$, then there are no codewords
in $\cB$ whose support is $\bldv$ since its distance from one of these codewords is at least $2(k-t)$.
If two codewords of $\cB$ whose support is $\bldu$ differ in exactly two coordinates of $\bldu \cap \bldv$, say the first and the second, then w.l.o.g.
these codewords are of the form $11 \cdots$ and $22 \cdots$ which implies that all the codewords in $\cB$
whose support is $\bldv$ are of the form $12 \cdots$ and $21 \cdots$ and hence there cannot be any codewords in $\cB$ whose support is $\bldu$
with another form.
Therefore, the number of distinct values for $\bldu \cap \bldv$ is at most $q-1$ which implies that
the number of codewords in $\cB$ whose support is $\bldu$ is at most $(q-1)^{k-t-1} \cdot (q-1) = (q-1)^{k-t}$.
Since by Theorem~\ref{thm:Wilson} we have that $\abs{\cC} \leq \binom{n-t-1}{k-t-1}$ for $n \geq (t+1)(k-t+1)$, it follows
that $\abs{\cB} \leq \binom{n-t-1}{k-t-1} (q-1)^{k-t} = \abs{\cA_q(t,k,n)}$ for large enough $n$.

\noindent
{\bf Case 2:} If $\cC_2 \neq \varnothing$ then
consider any $\bldu \in \cC_2$. Each codeword of $\cC$ intersects $\bldu$
in at least $t+2$ coordinates. Hence, we have that
$$
\abs{\cC} \leq \sum_{i=t+2}^k \binom{k}{i} \binom{n-k}{k-i}~.
$$
As in Case 1 we have that each codeword of $\cC_1$ is a support for at most $(q-1)^{k-t}$ codewords of~$\cB$. A codeword of $\cC_2$
can support at most $(q-1)^k$ codewords of $\cB$ since there are $q-1$ possible assignments for each of the $k$ coordinates of the support.
Hence, a codeword in $\cC$ can support at most $(q-1)^k$ codewords of $\cB$.
This implies that
\begin{equation}
\label{eq:bound_onB}
\abs{\cB} \leq (q-1)^k \sum_{i=t+2}^k \binom{k}{i} \binom{n-k}{k-i}  .
\end{equation}
$\sum_{i=t+2}^k \binom{k}{i} \binom{n-k}{k-i}$ is a polynomial
of degree $k-t-2$ in $n$, while $\abs{\cA_q (t,k,n)}= \binom{n-t-1}{k-t-1} (q-1)^{k-t}$ is
a polynomial in $n$ whose degree is $k-t-1$. Hence~(\ref{eq:bound_onB}) implies that
for large enough $n$ we have that $\cB$ is smaller than $\cA_q (t,k,n)$.

In both cases, we have proved that $\abs{\cB} \leq \abs{\cA_q(t,k,n)}$, i.e., $\cA_q(t,k,n)$ is a maximum size anticode
with diameter $2(k-t)-1$ for large enough~$n$.
\end{proof}

The anticode $\cA_q(t,k,n)$ is not a $t$-intersecting family of maximum size. The proof of Theorem~\ref{thm:asymA} fails
in this case since $\cC$ is a $(t+1)$-intersecting family and a $t$-intersecting family can be larger than $\cA_q(t,k,n)$.
Indeed, there exists such a $t$-intersecting family, which is $\cF_q(t,k,n)$.

We can present aa value of $n$ in Theorem~\ref{thm:asymA} such that the theorem is true for all values
above this $n$. The value should imply that $n$ is large enough for Case 2 or in another words
$$
(q-1)^k \sum_{i=t+2}^k \binom{k}{i} \binom{n-k}{k-i} < \binom{n-t-1}{k-t-1} (q-1)^{k-t} ~.
$$
This implies that it is enough to require $n > (q-1)^t (k-t-1)^2 \binom{k}{k/2} + 2k-t-2$ in Theorem~\ref{thm:asymA}.

Now, we define $\cA'_q(t,k,n) \triangleq \cF_q (t,k,n)$, i.e., $\cF_q (t,k,n)$ is defined as an anticode.
\begin{lemma}
\label{lem:propAT}
If $n \geq 2k-t$ then $\cA'_q (t,k,n)$ is an $(n,2(k-t),k)_q$ anticode
with $\binom{n-t}{k-t} (q-1)^{k-t}$ codewords.
\end{lemma}
\begin{proof}
Since $n \geq 2k-t$ and the weight in the last $n-t$ coordinates is at most $k-t$, it follows that there exist
two codewords in $\cA'_q (t,k,n)$ whose nonzero entries in the last $n-t$ coordinates are in disjoint coordinates.
For two such codewords the distance is $2(k-t)$ which is the diameter of the anticode.
The size of $\cA'_q (t,k,n)$ follows immediately by choosing the $k-t$ nonzero coordinates, each one
can be assigned with $q-1$ possible values (see also Lemma~\ref{lem:propF}).
\end{proof}

It will be proved now that similarly to $\cA_q(t,k,n)$ also $\cA'_q(t,k,n)$ is a maximum size anticode if $n$ is large enough.
The proof will be similar to the one of Theorem~\ref{thm:asymA}.

\begin{theorem}
\label{thm:asymAT}
If $n$ is large enough, then $\cA'_q (t,k,n)$ is a maximum size anticode for given $q > 2$, $k >t \geq 0$, and
diameter $2(k-t)$.
\end{theorem}
\begin{proof}
Let $\cB$ be an $(n,2(k-t),k)_q$ anticode and let $\cC$ be the set of supports of the codewords in~$\cB$, i.e.,
$$
\cC \triangleq \{ \support (\bldx) ~:~ \bldx \in \cB \}~.
$$
The code $\cC$ forms a $t$-intersecting family since the diameter $2(k-t)$ implies that each two codewords of $\cB$ must have
at most $2(k-t)$ positions in which one codeword has a nonzero element of $\Z_q$ and the second codeword has a \emph{zero}.

We partition $\cC$ into two subsets
\begin{align*}
\cC_1 & = \{ \bldv \in \cC ~:~ \abs{\bldv \cap \bldu} = t ~~ \text{for~some} ~~ \bldu \in \cC \}, \\
\cC_2 & = \{ \bldv \in \cC ~:~ \abs{\bldv \cap \bldu} \geq t+1 ~~ \text{for~all} ~~ \bldu \in \cC \}.
\end{align*}

We distinguish now between two cases depending whether $\cC_2$ is an empty subset or not.

\noindent
{\bf Case 1:} If $\cC_2 = \varnothing$ then $\cC=\cC_1$ and by Theorem~\ref{thm:Wilson} we have that $\abs{\cC} \leq \binom{n-k}{k-t}$.
Consider now two distinct codewords $\bldu, \bldv \in \cC$ whose intersection is of size $t$.
This already implies that their associated codewords in $\cB$ have distance $2(k-t)$ and hence their value on the coordinates
of $\bldu \cap \bldv$ are the same. This implies that
the number of codewords in $\cB$ whose support is $\bldu$ is at most $(q-1)^{k-t}$.
Hence, the number of codewords in $\cB$ is at most $\binom{n-t}{k-t} (q-1)^{k-t}$, i.e.,
for large enough $n$ we have that $\cB$ is smaller than $\cA'_q (t,k,n)$.

\noindent
{\bf Case 2:} If $\cC_2 \neq \varnothing$ then
consider any $\bldu \in \cC_2$. Each codeword of $\cC$ intersect $\bldu$
in at least $t+1$ coordinates. Hence, we have that the size of $\cC$ is at most
$$
\sum_{i=t+1}^k \binom{k}{i} \binom{n-k}{k-i}~.
$$
As in the proof of Theorem~\ref{thm:asymA} we have that each codeword of $\cC$ is a support for at most $(q-1)^k$ codewords of~$\cB$.
This implies that
\begin{equation}
\label{eq:bound_onBT}
\abs{\cB} \leq (q-1)^k \sum_{i=t+1}^k \binom{k}{i} \binom{n-k}{k-i}  .
\end{equation}
$\sum_{i=t+1}^k \binom{k}{i} \binom{n-k}{k-i}$ is polynomial of degree $k-t-1$ in $n$,
while $\abs{\cA'_q (t,k,n)}=\binom{n-t}{k-t} (q-1)^{k-t}$ is a polynomial in $n$ whose degree is $k-t$. Hence~(\ref{eq:bound_onBT}) implies that
for large enough $n$ we have that $\cB$ is smaller than $\cA'_q (t,k,n)$.

In both cases, we have proved that $\abs{\cB} \leq \abs{\cA'_q(t,k,n)}$, i.e., $\cA'_q(t,k,n)$ is a maximum size anticode
with diameter $2(k-t)$ for large enough~$n$.
\end{proof}

Similarly to Theorem~\ref{thm:asymA} we can show that for Theorem~\ref{thm:asymAT} it is enough to require that
$n > (q-1)^t (k-t)^2 \binom{k}{k/2} + 2k-t-1$.

We continue to prove that the anticodes $\cA_q (t,k,n)$ and $\cA'_q (t,k,n)$ are not just optimal, if $n$ is large enough then
they are unique for the fixed $q$, $t$, and $k$. We start with $\cA'_q (t,k,n)$ with a proof similar to the one
of Theorem~\ref{thm:uniqueInter}.

\begin{theorem}
\label{thm:uniqueAnti1}
If $n$ is large enough, then $\cA'_q (t,k,n)$ is unique maximum size anticode of length $n$, constant-weight $k$ and with diameter $2(k-t)$
over $\Z_q$.
\end{theorem}
\begin{proof}
Assume that in Theorem~\ref{thm:asymAT} $\abs{\cB} = \abs{\cA'_q(t,k,n)}$ and
$n$ is large enough. This implies that $\cC_2 = \varnothing$, $\cC = \cC_1$, and
$\abs{\cC_1}=\binom{n-t}{k-t}$.
By the Erd\"{o}s-Ko-Rado theorem (Theorem~\ref{thm:Wilson}), we know that $\cC$ consists
of all $k$-subsets that include a fixed $t$-subset of coordinates.
W.l.o.g., this $t$-subset is $\{1,\ldots,t\}$, and w.l.o.g.,
$\cB$ contains the word $\mathbf e = (\overbrace{1 \cdots 1}^{k ~ \text{times}} \overbrace{0 \cdots \cdots 0}^{n-k ~ \text{times}})$.
Now, any other codeword in~$\cB$ that has nonzero symbols in the first
$t$ coordinates and zeros in the next $k-t$ positions must start
with $t$ \emph{ones} (otherwise, the distance from $\mathbf e$ will be larger than $2(k-t)$).

Assume, to the contrary that there exists a codeword $\bldu$ in $\cB$ that does not start with $t$ \emph{ones}.
To have distance at most $2(k-t)$ between $\bldu$ and $\mathbf e$, $\bldu$ must have at least one nonzero
symbol in the next $k-t$ position. Now, consider a third codeword $\bldv \in \cB$ whose support intersects $\blde$ and $\bldu$ only
in the first $t$ positions. The codeword $\bldv$ must have the same values in the
first $t$ positions as $\blde$ and the same values in the first $t$ positions as $\bldu$, a contradiction.
Therefore, all the codewords in $\cB$ start with $t$ \emph{ones} which implies that $\cB$ is equivalent to $\cA'_q(t,k,n)$.
\end{proof}

The uniqueness of the anticode $\cA_q (t,k,n)$ is proved similarly to the uniqueness of the anticode $\cA'_q (t,k,n)$.
\begin{theorem}
\label{thm:uniqueAnti2}
If $n$ is large enough and $q>3$, then $\cA_q (t,k,n)$ is unique maximum size anticode of length $n$,
constant-weight $k$ and with diameter $2(k-t)-1$ over $\Z_q$.
\end{theorem}
\begin{proof}
Assume that in Theorem~\ref{thm:asymA} $\abs{\cB} = \abs{\cA_q(t,k,n)}$ and
$n$ is large enough. This implies that $\cC_2 = \varnothing$, $\cC = \cC_1$, and
$\abs{\cC_1}=\binom{n-t-1}{k-t-1}$. Since $\cC_1$ is a $(t+1)$-intersecting family, it follows by its size that
it is a maximum size family and by the Erd\"{o}s-Ko-Rado theorem (Theorem~\ref{thm:Wilson}), we know that $\cC$ consists
of all $k$-subsets that include a fixed $(t+1)$-subset of coordinates.
W.l.o.g., this $(t+1)$-subset is $\{1,\ldots,t,t+1\}$, and w.l.o.g.,
$\cB$ contains the word $\blde = (\overbrace{1 \cdots 1}^{k ~ \text{times}} \overbrace{0 \cdots \cdots 0}^{n-k ~ \text{times}})$.

Now, any other codeword in~$\cB$ that has nonzero symbols in the first
$t+1$ coordinates and zeros in the next $k-t-1$ coordinates must have at least
$t$ \emph{ones} in the first $t+1$ coordinates (otherwise, the distance from $\mathbf e$ will be larger than $2(k-t)-1$).
Since $q>3$ it follows from the structure of $\cC$ as explained in the proof of Theorem~\ref{thm:asymA} that all these
codewords have \emph{ones} in the first $t$ positions and a nonzero alphabet letter in the next position, where there are
codewords with each nonzero letter.

Assume there exists a codeword $\bldu \in \cB$ whose first $t$ coordinates have some values different from \emph{one} and a nonzero symbol $\alpha$
in position $t+1$. Let $\bldv$ another codeword in $\cB$ with \emph{ones} in the first $t$ positions and a nonzero symbol $\beta \neq \alpha$ in
position $t+1$. To avoid distance larger than $2(k-t)-1$ between $\bldu$ and $\bldv$, their supports must intersect in one more
coordinate outside the first $t+1$ coordinates. There exists a codeword $\bldz \in \cC$ whose support
does not intersect the support of $\bldu$ in the last $n-t-1$ coordinates and does not intersect the
support of $\bldv$ in the last $n-t$ coordinates. Moreover, $\bldz$ has a nonzero symbol $\gamma \notin \{ \alpha,\beta \}$ in position $t+1$.
To avoid distance larger than $2(k-t)-1$ between codewords,
this codeword $\bldz$ must have the same values as $\bldu$ in the first $t$ coordinates
and the same values as $\bldv$ in the first $t$ coordinates, a contradiction since $\bldu$ and $\bldv$ have different values in the
first $t$ coordinates. Therefore, all the codewords of $\cC$ have the same values on the first $t$ coordinates, w.l.o.g. \emph{ones}.

Thus, the unique maximum size $\cA_q(t,k,n)$ whose size is $\binom{n-t-1}{k-t-1} (q-1)^{k-t}$,
contains all the $\binom{n-t-1}{k-t-1} (q-1)^{k-t}$ words of weight $k$ over $\Z_q$ with \emph{ones}
in the first $t$ coordinates and a nonzero symbol in the $(t+1)$th coordinate.
\end{proof}

If $q=3$, then the same proof of Theorem~\ref{thm:uniqueAnti2} works as well. But, we need
to consider the case where all codewords of~$\cC$
whose support is~$\blde$ either begin with $t+1$ ones or
w.l.o.g. with $22$ and then $t-1$ ones
(there are $2\cdot(q-1)^{k-t-1}$ such words, which coincides with
$(q-1)^{k-t}$ for $q=3$).
For any support~$\bldv$ such that $\abs{\blde \cap \bldv}=t+1$,
only the words beginning with $12\overbrace{1\ldots1}^{t-1}$ or $21\overbrace{1\ldots1}^{t-1}$
can belong to~$\cB$. Since there are $2^{k-t}$ such words with support $\bldv$, $\mathcal B$ contains all of them.
Now, implying $n-t-1\ge3(k-t-1)$, we consider a third support $\bldz \in \mathcal C$ at maximum distance from both $\mathbf e$ and~$\mathbf v$, i.e., $\abs{\mathbf e \cap \bldz}=\abs{\bldv \cap \bldz}=t+1$. It is not difficult to verify that such codewords with support~$\bldz$ cannot exist
and the proof is completed.

\section{A Hierarchy of Anticodes}
\label{sec:hierarchy}

In this section we define two sequences of anticodes,
where each sequence has one anticode which was proved to be optimal asymptotically.
It will be proved that any two anticodes in the sequence are incomparable, i.e.,
one is larger for one range of alphabet sizes and lengths of the codewords
and the second is larger for the other alphabet sizes and lengths of the codewords.
For comparing two anticodes we should have that the alphabet size $q$,
the length of the words $n$, their weight $k$, and their diameter $D$ are the same for the two anticodes.

We start by defining the anticodes in these sequences. For $0 \leq \epsilon \leq k-t$, let
$$
\cA_q (t,\epsilon,k,n) \triangleq \{ ( \overbrace{1\cdots \cdots 1}^{t ~ \text{times}} a_1 \cdots a_\epsilon b_1 \cdots b_{n-t-\epsilon}) ~:~
a_i \in \Z_q \setminus \{0\}, ~  b_i \in \Z_q , ~ \weight{(b_1 \cdots b_{n-t-\epsilon})}=k-t-\epsilon \}
$$
The codewords of the anticode $\cA_q (t,\epsilon,k,n)$ have three parts: part $\A$ which consists of the first $t$ \emph{ones} in the codeword,
part $\C$ which consists of the last $n-t-\epsilon$ entries in the codewords, and part~$\B$ which consists of the middle $\epsilon$
entries in the codeword.

We note that $\cA_q (t,1,k,n) = \cA_q (t,k,n)$ and $\cA_q (t,0,k,n) = \cA'_q (t,k,n)$.
We have already proved in Theorem~\ref{thm:asymA} and Theorem~\ref{thm:asymAT}, respectively that
these anticodes are maximum size anticodes for odd diameter $D=2(k-t)-1$ and even diameter $D=2(k-t)$, respectively.

\begin{lemma}
\label{lem:diamHier}
If $n \geq 2k-t-\epsilon$, then the anticode $\cA_q (t,\epsilon,k,n)$ has diameter $2k-2t-\epsilon$.
\end{lemma}
\begin{proof}
Consider the following two codewords of $\cA_q (t,\epsilon,k,n)$,
$$
\bldc = ( ( \overbrace{1\cdots \cdots 1}^{t ~ \text{times}} \overbrace{1\cdots \cdots 1}^{\epsilon ~ \text{times}} \overbrace{1\cdots \cdots 1}^{k-t-\epsilon ~ \text{times}} \overbrace{0\cdots \cdots 0}^{n-k ~ \text{times}})
$$
and
$$
\bldc' = ( ( \overbrace{1\cdots \cdots 1}^{t ~ \text{times}} \overbrace{2 \cdots \cdots 2}^{\epsilon ~ \text{times}} \overbrace{0\cdots \cdots 0}^{n-k ~ \text{times}} \overbrace{1\cdots \cdots 1}^{k-t-\epsilon ~ \text{times}} ).
$$
Since $n \geq 2k-t-\epsilon$, i.e., $n-t-\epsilon \geq 2(k-t-\epsilon)$, it follows that
$$
\distance(\bldc,\bldc')= \epsilon + 2(k-t-\epsilon)= 2k-2t-\epsilon.
$$
These two codewords are of maximum distance in the anticode $\cA_q (t,\epsilon,k,n)$.
Hence, the anticode $\cA_q (t,\epsilon,k,n)$ has diameter $2k-2t-\epsilon$.
\end{proof}

\begin{lemma}
\label{lem:anti_size}
The size of the anticode $\cA_q (t,\epsilon,k,n)$ is $\binom{n-t-\epsilon}{k-t-\epsilon} (q-1)^{k-t}$.
\end{lemma}
\begin{proof}
The number of possible combinations in part $\B$ is $(q-1)^\epsilon$ and the number of the combinations in part $\C$ is
$\binom{n-t-\epsilon}{k-t-\epsilon} (q-1)^{k-t-\epsilon}$, which implies the claim in the lemma.
\end{proof}

Since the diameter $D$ of the anticode $\cA_q (t,\epsilon,k,n)$ is $2k-2t-\epsilon$, it follows that this diameter
can be even or odd, depending whether $\epsilon$ is even or odd, respectively.
For $\cA_q (t,\epsilon,k,n)$ we have that the diameter is $D=2k-2t-\epsilon$ and this anticode
will be compared with the anticode $\cA_q (t',\epsilon',k,n)$,
where $\epsilon' = \epsilon+2$, $t'=t-1$, and hence its diameter is $D'=D$.

\begin{lemma}
\label{lem:larger_smaller}
When $n \geq 2k -t - \epsilon$, the anticode $\cA_q (t',\epsilon',k,n)=\cA_q (t-1,\epsilon+2,k,n)$ is larger than $\cA_q (t,\epsilon,k,n)$ if and only if
$q > \frac{n+k-2t-2\epsilon}{k-t-\epsilon}$ or equivalently $n < (q-2)(k-t-\epsilon)+k$ (both anticodes have diameter $2k-2t-\epsilon$).
\end{lemma}
\begin{proof}
By Lemma~\ref{lem:diamHier} the diameter of the two anticodes is $2k-2t-\epsilon$.
By Lemma~\ref{lem:anti_size} the size of $\cA_q (t,\epsilon,k,n)$ is $\binom{n-t-\epsilon}{k-t-\epsilon} (q-1)^{k-t}$ and the size of
$\cA_q (t',\epsilon',k,n)$ is $\binom{n-t-\epsilon-1}{k-t-\epsilon-1} (q-1)^{k-t+1}$. Therefore, $\cA_q (t',\epsilon',k,n)$ is larger than
$\cA_q (t,\epsilon,k,n)$ if and only if
$$
\binom{n-t-\epsilon-1}{k-t-\epsilon-1} (q-1)^{k-t+1} > \binom{n-t-\epsilon}{k-t-\epsilon} (q-1)^{k-t}
$$
which is equivalent to
$$
q-1 > \frac{n-t-\epsilon}{k-t-\epsilon}
$$
or
$$
n < (q-2)(k-t-\epsilon)+k ~.
$$
\end{proof}

Lemma~\ref{lem:larger_smaller} implies that we can have two sequences of incomparable anticodes for given
$t$, $k$, and $D=2(k-t)$. For any given $q$ if $n$ is large enough then $\cA_q (t,0,k,n)$ is the maximum size
anticode. However, if $q > \frac{n+k-2t}{k-t}$, then the anticode $\cA_q (t-1,2,k,n)$ is larger
than $\cA_q (t,0,k,n)$. If $q > \frac{n+k-2(t-1)-2}{k-(t-1)-1}$, then the anticode $\cA_q (t-2,4,k,n)$ is larger
than $\cA_q (t-1,2,k,n)$ and as a result also larger than the anticode $\cA_q (t,0,k,n)$.
We can continue and obtain a sequence of $\text{minimum}\{t+1,k-t+1\}$ anticodes (the minimum is due to the fact
that $k-t-\epsilon$ cannot be negative),
where the first anticode is of maximum size if $n$ is large enough.
If $n$ is fixed, then from a certain alphabet size the second anticode is larger.
Similar hierarchy can be defined for odd diameter and the anticodes $\cA_q (t,\epsilon,k,n)$, where $\epsilon$ is odd.
The anticode $\cA_q (t,1,k,n)$ is asymptotically of maximum size and each two anticodes in this hierarchy are incomparable.



\section{Conclusion, Discussion, and Future Research}
\label{sec:conclusion}

This paper considers maximum size anticodes and maximum size $t$-intersecting families over a non-binary alphabet.
Maximality and uniqueness of such anticodes was proved and hierarchy between anticodes
was given. Such different anticodes were discussed also in other papers as follows.

The intersecting family $\cF_q (t,k,n)$ which was proved to be a maximum size $(t,k)_q$-intersecting family
if $n \geq 2k-t$ is large enough (and it was referred to as $\cA'_q(t,k,n)$ or $\cA_q(t,0,k,n)$) was already
defined in~\cite{Etz22}, where it was proved to be a maximum size anticode if a certain structure called a generalized Steiner
system with appropriate parameters exists.

There are other anticodes which were proven to be of maximum size in~\cite{Etz22}. For $1 \leq \epsilon \leq k$, the anticode
$$
\{ ( \bldb_1 \cdots \bldb_{\epsilon} \overbrace{1\cdots \cdots 1}^{{t=k-\epsilon} ~ \text{times}}  \overbrace{0\cdots \cdots 0}^{{n-k} ~ \text{times}} ) ~:~
\bldb_i \in \Z_q \setminus \{0\}, ~ 1 \leq i \leq \epsilon \}.
$$
This anticode is exactly $\cA_q(t,\epsilon,k,n)$, where $t=k-\epsilon$, in our hierarchy.
It is a maximum size anticode when certain structures exist (see~\cite{Etz22} for more details). These structures exist
for many parameters. This anticode is an $(n,\epsilon,k)$ anticode with $(q-1)^\epsilon$ codewords.

Two more anticodes, which are of maximum size for certain parameters, were defined in~\cite{SXK23}. The first one
$$
\{ ( \blda_1 \cdots \blda_\epsilon \bldb_1 \cdots \bldb_{n-\epsilon} ) ~:~
\blda_i \in \Z_q \setminus \{0\}, ~ \wt{\bldb_1 \cdots \bldb_{n-\epsilon}}=k-\epsilon \}
$$
This anticode is exactly $\cA_q(0,\epsilon,k,n)$ in our hierarchy.
It is an $(n,2k-t,k)$ anticodes with $\binom{n-t}{k-t} (q-1)^k$ codewords. This anticode is of maximum size
if $n \geq (k-t+1)(t+1)$ and $q$ is large enough.

The second one
$$
\{ ( \overbrace{0\cdots \cdots 0}^{t ~ \text{times}}  \bldb_1 \cdots \bldb_{n-t} ) ~:~
\wt{\bldb_1 \cdots \bldb_{n-t}}=k \}
$$
is an $(n,n-t,k)$ anticode with $\binom{n-t}{k} (q-1)^k$ codewords. This anticode is of maximum size if $n \leq (k+t-1)(t+1)/t$
and $q$ is large enough.

Other maximum size anticodes mentioned in~\cite{Etz22} are not relevant for our discussion since either $n=k$ or $n=k+1$.
As we see, there are many maximum size anticodes with similar parameters and other large anticodes which we did not prove
their maximality. Each maximum size anticode is of maximum size in different parameters and there are other large
anticodes in different lengths and alphabet size which cannot be compared (each one is larger in different parameters).
Some comparisons and hierarchy between the anticodes were given in this paper and there are more comparisons in~\cite{Etz22},
where also the uniqueness of other anticodes for some given parameters is discussed.
As for future research we would like to prove the maximality of all the anticodes in the hierarchies,
We also would like to know whether such hierarchies exist also for intersecting families.


%

\end{document}